\documentclass[10pt]{article}

\usepackage{a4wide}
\usepackage{amsmath}
\usepackage{amssymb}
\usepackage{stmaryrd}
\usepackage{latexsym}
\usepackage{amsthm}
\usepackage[cmtip,arrow]{xy}
\usepackage{pb-diagram,pb-xy}

\usepackage{bm}
\usepackage{bbm}

\usepackage[latin1]{inputenc}
\usepackage{xy}
\xyoption{all}

\theoremstyle{plain}

\numberwithin{equation}{section}

\newtheorem{thm}[equation]{Theorem}
\newtheorem{lem}[equation]{Lemma}
\newtheorem{prop}[equation]{Proposition}
\newtheorem{cor}[equation]{Corollary}
\newtheorem{defn}[equation]{Definition}

\newcommand{\bN}{\mathbb{N}}
\newcommand{\bR}{\mathbb{R}}
\newcommand{\bX}{\mathbb{X}}
\newcommand{\bY}{\mathbb{Y}}
\newcommand{\bZ}{\mathbb{Z}}

\newcommand{\cA}{\mathcal{A}}
\newcommand{\cB}{\mathcal{B}}
\newcommand{\cC}{\mathcal{C}}
\newcommand{\cF}{\mathcal{F}}
\newcommand{\cI}{\mathcal{I}}
\newcommand{\cJ}{\mathcal{J}}
\newcommand{\cK}{\mathcal{K}}
\newcommand{\cS}{\mathcal{S}}
\newcommand{\cU}{\mathcal{U}}

\newcommand{\fX}{\mathfrak{X}}
\newcommand{\fY}{\mathfrak{Y}}
\newcommand{\fZ}{\mathfrak{Z}}

\newcommand{\colim}{\operatorname{colim}}

\newcommand{\Ho}{\operatorname{Ho}}

\newcommand{\id}{\operatorname{id}}
\newcommand{\map}{\operatorname{map}}

\newcommand{\supp}{\operatorname{supp}}

\newcommand{\Ob}{\mathfrak{Ob}}
\newcommand{\dash}{^{\prime}}
\newcommand{\ttop}{\mbox{\textbf{Top}}}
\newcommand{\topstacks}{\textbf{TopStacks}}
\newcommand{\hq}{/\hspace{-1.2mm}/}
\newcommand{\sk}{\operatorname{sk}}
\newcommand{\st}{\operatorname{st}}
\newcommand{\inc}{\operatorname{\cJ}}

\begin{document}

\title{The homotopy type of a topological stack}

\author{
Johannes Ebert,\\
Mathematisches Institut der Universit\"at Bonn\\
Beringstra{\ss}e 1\\
53115 Bonn, Germany\\
ebert@math.uni-bonn.de}

\maketitle

\begin{abstract}
The notion of the \emph{homotopy type} of a topological stack has
been around in the literature for some time. The basic idea is that
an atlas $X \to \fX$ of a stack determines a topological groupoid
$\bX$ with object space $X$. The homotopy type of $\fX$ should be
the classifying space $B \bX$. The choice of an atlas is not part of
the data of a stack and hence it is not immediately clear why this
construction of a homotopy type is well-defined, let alone
functorial. The purpose of this note is to give an elementary
construction of such a homotopy-type functor.
\end{abstract}

\section{Introduction}

The concept of a stack (which originated in algebraic geometry)
plays an increasingly important role in geometric topology, see for
example \cite{BGNX}, \cite{FHT}, \cite{EG}, \cite{EG2}. In this note
we show how a stack defines an object of homotopy theory.

We assume that the reader is familiar with the terminology of
stacks. Therefore we will not spell out the basic definitions here.
A stack over the site $\ttop$ of topological spaces is a lax sheaf
of groupoids on the site of topological spaces; we refer the reader
to the excellent \cite{Hein} an explanation of this definition. Stacks over $\ttop$ form a $2$-category all of whose $2$-morphisms are isomorphisms.
There are several possible notions of \emph{topological stacks}. Our
notion is made explicit in \ref{deftopstack}. Essentially, a
topological stack is a stack over $\ttop$ which can be represented
by a topological groupoid. Let $\topstacks$ denote the $2$-category
of topological stacks. For any $2$-category $ \cS$, we denote the
underlying ordinary category by the symbol $\tau_{\leq 1} \cS$.

We would like to construct a functor $\Ho: \tau_{\leq 1} \topstacks
\to \ttop$ which assigns to a stack its homotopy type. For
set-theoretical reasons, we need to restrict to small subcategories
of $\cS \subset \topstacks$. This level of sophistication is
certainly sufficient for all applications of our construction to
concrete mathematical problems.

Furthermore, it turns out that we need to restrict to stacks which
admit a presentation by a ``paracompact groupoid'' (see below for
details). This is a rather mild condition, which is satisfied by
virtually all stacks of interest in geometric topology. In the sequel, we
assume that $\cS$ is a small $2$-category and there is a fixed
$2$-functor $\inc:\cS \to  \topstacks$ such that all stacks in the image
of this functor admit a presentation by a paracompact groupoid.

The first main result of this note is

\begin{thm}\label{mainthm2}
There exists a functor $\Ho : \tau_{\leq 1} \cS \to \ttop$, which
assigns to $s \in \Ob (\cS)$ a topological space $\Ho(\fX)$ which is
homotopy equivalent to $B \bX$, when $ \bX$ is a groupoid presenting the topological 
$\fX=  \cJ (s)$. If $f,g$ are two $1$-morphisms with the same source and
target, then $\Ho(f)$ and $\Ho(g)$ are homotopic if $f$ and $g$ are
$2$-isomorphic or if $\inc (f)$ and $\inc (g)$ are concordant (see
Definition \ref{defconcord} below).
\end{thm}

Let $\pi_0 (\cS)$ be the category which is obtained from $\tau_{\leq
1} \cS$ by identification of $2$-isomorphic $1$-morphisms; there is
a quotient functor $\tau_{\leq 1} \cS \to \pi_0 (\cS)$. Note that
the fully faithful Yoneda embedding $\st:\ttop \to \topstacks$
defines a fully faithful functor $\ttop \to \pi_0 \topstacks$ -
homotopic but different maps of spaces do not yield $2$-isomorphic
morphisms of stacks. Furthermore, let $\Ho \ttop$ be the homotopy
category of topological spaces. As a corollary of Theorem
\ref{mainthm2}, we obtain the existence of a homotopy type functor
$\pi_0 (\cS) \to \Ho \ttop$. This homotopy type functor extends both, the obvious functor $\ttop \to \Ho \ttop$ and the functor from topological groups to $\Ho \ttop$ sending $G $ to $BG$.

There is an essential feature of homotopy types which is abandoned
in Theorem \ref{mainthm2}. Let us describe what is missing. Let $X$
be a space. Then we denote, as usual, the stack $\st(X)$ by the
symbol $X$; there is no danger of confusion. The space
$\Ho(\fX)$ should come with a map $\eta_{\fX}:\Ho(\fX) \to \fX$
which should be a \emph{universal weak equivalence}, i.e. for any
space $Y$ and any $Y \to \fX$, the pullback $Y \times_{\fX} \Ho(\fX)
\to Y$ is a weak homotopy equivalence (the morphism $Y \to \fX$ is
automatically representable by \cite{Noo1}, Corollary 7.3; thus $Y
\times_{\fX} \Ho(\fX)$ is a topological space). The map $\eta_{\fX}$ is a generalization of the map $BG \to \ast \hq G$ given by the universal principal $G$-bundle.

Obviously, it is desirable that the maps $\eta_{\fX}$ assemble to a
natural transformation $\eta: \st \circ \Ho \to \tau_{\leq 1} \inc$
of functors $\tau_{\leq 1} \cS \to \tau_{\leq 1} \topstacks$. We were not able to
construct such a natural transformation on the nose, but only
\emph{up to contractible choice} and up to $2$-isomorphism. The following two definition make these notions precise.

\begin{defn}
Let $\cC$ be a (discrete, small) category. A \emph{functor defined
up to contractible choice} is a triple $(\widetilde{\cC},p,F)$,
where $\widetilde{\cC}$ is a topological category which has the same
objects as $\cC$ , $p: \widetilde{\cC} \to \cC$ is a functor which
is the identity on objects and a weak homotopy equivalence on
morphism spaces (i.e. $\widetilde{\cC}$ is a thickening of $\cC$ in
the sense that the morphisms in $\cC$ are replaced by contractible spaces of
morphisms) and $F: \widetilde{C} \to \ttop$ is a continuous functor.\\
\end{defn}

\begin{defn}
Let $\cA$ be a topological category with discrete object set and
$\cB$ be a discrete $2$-category all of whose $2$-morphisms are
isomorphisms. Let $F_0, F_1: \cA \to \cB$ be two functors. A
\emph{pseudo-natural transformation} $\eta$ assigns to every object
$a \in \cA$ a $1$-morphism $\eta_a:F_0 (a) \to F_1 (a)$ and to every
morphism $f:a \to a \dash$ of $\cA$ a $2$-isomorphism $\eta_f : F_1
(f) \circ \eta_a \to \eta_{a \dash} \circ F_0 (f)$ such that
$\eta_{\id_a}=\id_{\eta_a}$ and such that for any pair $f,f \dash$
of composable morphisms, the $2$-isomorphisms $\eta_f, \eta_{f
\dash} $ and $\eta_{f \dash \circ f}$ are compatible.
\end{defn}

For most (but not all) constructions of homotopy theory, a functor
defined up to contractible choice is as good as an honest functor.
Therefore the following theorem should be a satisfactory result for
many purposes.

\begin{thm}\label{mainthm1}
Let $\cS$ be as above. Then there exists a functor defined up to
contractible choice $(\widetilde{\tau_{\leq 1} \cS}, p,
\widetilde{\Ho})$. The functor $\widetilde{\Ho}$ is related to the
functor $\Ho \circ p$ from Theorem \ref{mainthm2} by a zig-zag of
natural transformations which are weak homotopy
equivalences on each object.

Moreover, there exists a pseudo-natural transformation $\eta: \st
\circ \widetilde{\Ho} \to \inc \circ p$ of functors
$\widetilde{\tau_{\leq 1} \cS}\to  \topstacks$. For any
stack $\fX \in \cS$, the morphism $\eta_{\fX}: \widetilde{\Ho} (\fX)
\to \fX$ is a \emph{universal weak equivalence}, in the sense that
for any space $Y$ and any $Y \to \fX$, the pullback $Y \times_{\fX}
\Ho(\fX) \to Y$ is a weak homotopy equivalence.
\end{thm}

The results of the present paper are very similar to those of
Behrang Noohi's recent paper \cite{Noo2} (in fact, they are slightly
weaker). Proposition 11.2 in \cite{Noo2} implies \ref{mainthm1}. On
the other hand our treatment is more elementary. Therefore we claim
that the present paper should be useful for anyone whose main
interest is in the \emph{applications} of topological stacks to
problems in geometry.

However, there is a flaw in the theory (also in \cite{Noo2}) which
we will describe now. Given a homotopy invariant functor $F$ from
spaces to groups (or any other discrete category), we can extend $F$
to stacks via

\begin{equation}\label{extensionoffunctors}
\hat{F}(\fX):= F(\Ho(\fX)).
\end{equation}

It follows that $\hat{F}(\eta_{\fX})$ is an isomorphism. For many
of the functors of algebraic topology, including singular (co)homology or
homotopy groups, this is a reasonable definition. But there are
important homotopy-invariant functors on spaces for which
\ref{extensionoffunctors} is \emph{not} a good definition. The prime
example is complex $K$-theory. Any reasonable definition of the
$K$-theory of a stack should satisfy $K^0 (\ast \hq G) \cong RG$ for
a compact Lie group $G$ ($RG$ is the representation ring). In fact
the definition of the $K$-theory of a stack given in \cite{FHT}
satisfies this condition. On the other hand, the celebrated
Atiyah-Segal completion theorem \cite{AS} shows, among other things,
that the natural map $RG \to K^0 (BG)$ is not an isomorphism. It
follows that, in the present setup, $K$-theory of a stack is not
homotopy-invariant. Gepner and Henriques \cite{GH} developed a finer
homotopy theory of stacks in which $K$-theory is homotopy invariant.
On the other hand, their theory is much more involved than ours.

Here is a brief outline of the paper. In section
\ref{groupoidsandstacks}, we discuss the notion of a
principal $\bX$-bundle for an arbitrary topological groupoid $\bX$ and
define the stack $\bX_0 \hq \bX_1$ of principal $\bX$-bundles, which is
the prototype of a topological stack. Then we construct the
universal principal $\bX$-bundle. The material in this section is a
rather straightforward generalization of the classical theory of
classifying spaces for topological groups. However, we need to be
precise on the point-set level, and Theorem
\ref{classprinbun} is stronger than what is standard in the theory
of fibre bundles. Section \ref{proofs} contains the proofs of
Theorem \ref{mainthm2} and \ref{mainthm1}. In an appendix
\ref{anhang}, we show a technical result which guarantees the
paracompactness of the classifying space of a ``paracompact
groupoid''.

\subsection*{Acknowledgements}

The author was supported by a postdoctoral grant from the German
Academic Exchange Service (DAAD). He enjoyed the hospitality of the
Mathematical Institute of the University of Oxford. He is
particularly grateful to Jeff Giansiracusa for numerous discussions
about stacks. He also wants to thank the anonymous referee who
pointed out a mistake at a crucial step of the argument in an
earlier version of this paper.

\section{Groupoids and stacks}\label{groupoidsandstacks}

\subsection{Principal bundles for groupoids}

Let $\bX= (\bX_0, \bX_1, s,t,e,m, \iota)$ be a topological groupoid:
$\bX_0$ is the object space, $\bX_1$ the morphism space, $s,t: \bX_1
\to \bX_0$ are source and target maps, $m: \bX_1 \times_{\bX_0}
\bX_1 \to \bX_1$ is the multiplication, $e: \bX_0 \to \bX_1$ is the
unit map and $\iota: \bX_1 \to \bX_1$ sends a morphism to its
inverse. The maps $s,t,e,m, \iota$ are continuous and satisfy the
usual identities. We use the following convention: the product
$xy=m(x,y)$ of $x,y \in \bX_1$ is defined if (and only if)
$t(x)=s(y)$.

We say that an $\bX$-space over a space $X$ consists of a space $E$,
two maps $p:E \to X$ and $q:E \to \bX_0$ and an ''action'' $\alpha:E
\times_{\bX_0} \bX_1 = \{ (e, \gamma) \in E \times \bX_1 | q(e) =
s(\gamma) \} \to E$ over $X$ which is compatible with the
projection to $X$ and the multiplication in $\bX$. In other words, the diagrams

\[
\xymatrix{E \times_{\bX_0} \bX_1 \ar[r]^{\alpha} \ar[d]^{p_1} & E \ar[d]^{p}\\
E \ar[r]^{p} & X
}
\]

($P_1$ is the projection onto the first factor) and

\[
\xymatrix{E \times_{\bX_0} \bX_1 \times_{\bX_0} \bX_1 \ar[r]^{(\id,m)} \ar[d]^-{(\alpha,\id)} & E \times_{\bX_0} \bX_1 \ar[d]^{\alpha}\\
E \times_{\bX_0} \bX_1 \ar[r]^{\alpha} & E
}
\]

are required to commute.
Isomorphisms, pullbacks and restrictions of $\bX$-spaces over $X$
are defined in the obvious way. We will often write an $\bX$-space
shortly as $p:E \to X$, with the maps $q$ and $\alpha$ understood.

To define principal $\bX$-bundles, we start with trivial bundles. The diagram

\begin{equation}\label{localmodel}
\xymatrix{
\bX_1 \ar[r]^{t} \ar[d]^{s} & \bX_0\\
\bX_0 & \\
}
\end{equation}

defines an $\bX$-space over $\bX_0$ ($p:=s$, $q:=t$, $\alpha:= m$).
This serves as the local model for a principal $\bX$-bundle.

\begin{defn}\label{defprinbun}
Let $\bX$ be a topological groupoid and let $X$ be a topological
space. A \emph{principal $\bX$-bundle} on $X$ is an $\bX$-space
$(E; p, q, \alpha)$ over $X$ which is locally trivial in the sense
that each $x \in X$ has an open neighborhood $U \subset X$ which
admits a map $h:U \to \bX_0$ and an isomorphism of the restriction
$E|_{U}$ with the pullback of \ref{localmodel} via $h$.
\end{defn}

There is an equivalent notion which is more abstract but also more
common in the theory of stacks. It is the notion of
\emph{$\bX$-torsors}. We will not use this notion, but we explain
it briefly. Given any map $T \to X$ of spaces, we can form the
groupoid $X_T$ with object space $T$ and morphism space $T \times_X
T$. We say that a morphism of topological groupoids $\bX \to \bY$ is
called \emph{cartesian} if the diagrams

\[
 \xymatrix{
\bX_1 \ar[r] \ar[d]^{s} & \bY_1  \ar[d]^{s} & & \bX_1 \ar[r] \ar[d]^{t} & \bY_1  \ar[d]^{t} \\
\bX_0 \ar[r] & \bY_0 & & \bX_0 \ar[r] & \bY_0 \\
}
\]

are cartesian. An \emph{$\bX$-torsor} on a space $X$ is a map $p:T \to
X$ which admits local sections together with a cartesian morphism of
groupoids $\phi:X_T \to \bX$.

Given such an $\bX$-torsor, a principal $\bX$-bundle is given as
follows: The total space $E$ is just $T$ and $\phi_0$ gives a map $E
\to \bX_0$. The action $\alpha$ is the composition $T \times_{\bX_0}
\bX_1 \cong T \times_X T \to T$ (projection onto the first factor).

Conversely, let $(E, p, q, \alpha)$ be a principal $\bX$-bundle.
Put $T := E$. The map $p$ has local sections by the definition of a
principal bundle\footnote{Note that $\iota: \bX_0 \to \bX_1$ is a
section to the map $s$.}. Let $e_1, e_2 \in E$ be two points in the
same fibre, $p(e_1) = p(e_2)$, i.e. $(e_1, e_2) \in E \times_X E$.
There exists a unique $\gamma_{(e_1,e_2)} \in \bX_1$ such that
$\alpha(e_1, \gamma_{(e_1,e_2)}) = e_2$. Assigning $(e_1, e_2)
\mapsto \gamma_{(e_1,e_2)}$ defines a map $\phi_1:E \times_X E \to
\bX_1$ which is continuous by the local triviality of a principal
bundle. This map fits into a commutative diagram

\[
 \xymatrix{
T \times_X T \ar[r]^{\phi} \ar[d]^{s,t} & \bX_{1}  \ar[d]^{s,t} \\
T \ar[r]^{q} & \bX_{0} ,
}
\]

which is easily seen to be cartesian. We will not use the concept of
a torsor any more in this note. One can also describe principal $\bX$-bundles in terms of transition functions, see
\cite{Haef}.

It is worth to spell out what a principal $\bX$-bundle is for
familiar groupoids. Let $Y$ be a space, considered as a groupoid
$\bX$ without nontrivial morphisms. Let $(E,p,q,\alpha)$ be a principal $\bX$-bundle on $X$. The maps in \ref{localmodel} are
identities; thus $p:E \to X$ is a homeomorphism; $q \circ p^{-1}$ is
a map $X \to Y$, and under these identifications $\alpha$ is the
canonical homeomorphism $X \times_{X} X \to X$. Thus a principal
bundle for the trivial groupoid is the same as a continuous map $X
\to Y$.

A similar situation is met when $\bX=Y_{\cU}$ is the groupoid
associated to an open cover $\cU$ of the space $Y$ (see \cite{seg}).
The map $E \to X$ is then an open cover of $X$, $q$ is a collection
of locally defined maps to $Y$ (in fact, to the elements of the
original open cover). The map $\alpha$ is precisely the information
that all these locally defined maps fit together to form a globally
defined map $X \to Y$.

If $G$ is a topological group, considered as a groupoid with one
object, then the notion of principal $G$-bundle from Definition
\ref{defprinbun} agrees with the traditional notion of a principal
bundle, with the exception that we assume that $G$ acts from the
\emph{left} on the total space.

Similarly, if $G$ acts on the space $Y$ from the left, we form the
groupoid $G \int Y$: the object space is $Y$, the morphism space is
$G \times Y$ and the action is used for the structural map of a
groupoid. A principal $G \int Y$-bundle on a space $X$ consists of a
principal $G$-bundle $P \to X$ and an $G$-equivariant map $q:P \to
Y$.

\subsection{The stack of principal $\bX$-bundles}\label{stackprinbun}

The collection of all principal $\bX$-bundles on a space $X$,
together with their isomorphisms, forms a groupoid which we denote by $\bX_0
\hq \bX_1 (X)$. The functor $X \mapsto \bX_0 \hq \bX_1 (X)$ is a
(lax) presheaf of groupoids on the site of topological spaces. The local
nature of principal bundles shows that this is actually a
\emph{sheaf} of groupoids, in other words, a stack, which we denote
by $\bX_0 \hq \bX_1$.

Clearly, morphisms of groupoids yield morphisms of stacks and
natural transformations of morphisms give $2$-morphisms of stacks.
Thus there is a $2$-functor from topological groupoids to stacks,
sending $\bX$ to $\bX_0 \hq \bX_1$. This $2$-functor is far from
being an equivalence of categories. There can be a morphism
$\phi:\bX \to \bY$ of groupoids such that the induced morphism
$\phi_*: \bX_0 \hq \bX_1 \to \bY_0 \hq \bY_1$ is an equivalence, but
$\phi$ does not have an inverse. This is not particularily exotic:
given a topological group $G$ and a principal $G$-bundle $P \to X$,
then the obvious groupoid morphism $G \int P \to X$ induces an
equivalence of stacks $P \hq G \cong X$, but there is no inverse
unless $P$ is trivial. A similar situation is met when $\cU$ is an
open cover of the space $X$ which defines a groupoid $X_{\cU}$ and
an equivalence $X_{\cU} \to X$. It is essential for the theory of
stacks to allow for inverses of these morphisms.

It may seem that the stacks $\bX_0 \hq \bX_1$ are quite special, but
this is not the case. The data of a stack $\fX$, a space $X$ and a
representable map $\varphi:X \to \fX$ determines a topological
groupoid $\bX$. Namely, $\bX_0 = X$, $\bX_1 = X \times_{\fX} X$; the
structure maps for a groupoid are easy to find and the proof of the
groupoid axioms is easy as well.

Moreover, the map $\varphi$ determines a morphism
$\hat{\varphi}:\bX_0 \hq \bX_1 \to \fX$ of stacks. It is easy to see
that $\varphi$ is a chart (in the sense of \cite{Noo1}, Def. 7.1) if
and only if $\hat{\varphi}$ is an equivalence of stacks.

\begin{defn}\label{deftopstack}
A stack $\fX$ over the site $\ttop$ is a topological stack if there
exists a topological groupoid $\bX$ and an equivalence of stacks
$\bX_0 \hq \bX_1 \to \fX$.
\end{defn}

This notion agrees with the notion of a ''pretopological stack''
defined in \cite{Noo1}, Def. 7.1.

\subsection{The universal principal bundle}

We now want to define the universal principal $\bX$-bundle. The
topological category $\bX \downarrow \bX$ is the category of arrows
in $\bX$; more precisely, an object is an arrow $\gamma:x \to y$ in
$\bX$ and a morphism from $(\gamma: x\to y)$ to $(\gamma \dash: x
\dash \to y)$ is a morphism $\delta:x \to x \dash$
with\footnote{Sic; remember our convention about multiplication in a
groupoid.} $ \delta  \gamma \dash := m(\delta,\gamma \dash) =
\gamma$; there is no morphism $(\gamma: x\to y) \to (\gamma \dash: x
\dash \to y\dash)$ if $y \neq y \dash$.

There is a topology on $\bX \downarrow \bX$ induced from the
topology on $\bX$ and the functor $p:\bX \downarrow \bX \to \bX$; $p
(\gamma: x \to y)=x$, $p(\delta) = \delta$ is continuous. Similarly,
the functor $\zeta:\bX \downarrow \bX \to \bX_0$ (the trivial
groupoid with object space $\bX_0$) which sends $\gamma: x \to y$ to
$y$ is continuous.

In this note we say that the classifying space of a topological
category $\cC$ is the \emph{thick} geometric realization $\|
N_{\bullet} \cC \|$ of the nerve $N_{\bullet} \cC$ of the category.
The thick realization $\|X_{\bullet}\|$ of a simplicial space
$X_{\bullet}$ is the space
\[
 \| X_{\bullet} \| := \coprod_{n \geq 0} X_n \times \Delta^n / \sim,
\]
where $\sim$ is the equivalence relation generated by $(\varphi^*
x,t) \sim (x, \varphi_* t)$ for any \emph{injective} map $\varphi$
in the simplex category $\triangle$.

We apply the classifying space construction to the functors $p$ and
$\zeta$ and obtain a diagram

\begin{equation}\label{univprinbun}
\xymatrix{B (\bX \downarrow \bX) \ar[r]^-{B \zeta} \ar[d]^{B p} & \bX_0\\
B \bX & \\
}
\end{equation}
and, furthermore, a map $B (\bX \downarrow \bX) \times_{\bX_0} \bX_1
\to B (\bX \downarrow \bX)$ which is given by multiplication. We
will abbreviate $E \bX := B (\bX \downarrow \bX)$. In the appendix,
we shall describe two other convenient models for the spaces $B \bX$
and $E \bX $.

\begin{prop}\label{isprinbun}
The diagram \ref{univprinbun} is a principal $\bX$-bundle.
\end{prop}

The proof is a straightforward generalization of Milnor's
construction of classifying spaces for topological groups. It is
given in \cite{Haef} and also in \cite{Noo2}. The proofs uses a
different description of the topological space $B \bX$. We say a few
words about the latter point in the appendix.

\begin{prop}\label{lemma1}
There exists a section $\sigma: \bX_0 \to E \bX$ of $B \zeta$ and a
deformation retraction of $E \bX $ onto $\sigma(\bX_0)$ over
$\bX_0$. Consequently, the space of sections of $B \zeta$ is
contractible.
\end{prop}

\begin{proof}
Let $\sigma: \bX_0 \to \bX \downarrow \bX$ be the functor which
sends $x$ to $\id: x \to x$. Clearly $\zeta \circ \sigma= \id$ and
there is an evident natural transformation $T:\id_{\bX \downarrow
\bX} \to \sigma \circ \zeta$. The composition $T \circ \sigma$ is
the identity transformation. So after realization, $\sigma$ defines
the desired section and $T$ defines the deformation retraction.

The well-known theorem that a continuous natural transformation
between two functors of topological categories defines a homotopy
between the maps on classifying spaces still holds if the geometric
realization of the nerve is replaced by the thick geometric
realization, except that we now get a contractible space of
preferred homotopies instead of just one (combine the standard proof
of the theorem in \cite{seg} with \cite{segcat}, Prop. A.1 (iii) and
an explicit computation of $\| N_{\bullet} (0 \to 1) \|$).
\end{proof}

\begin{thm}\label{classprinbun}
Let $E^{univ} \to B^{univ}$ be a principal $\bX$-bundle such that
the space of sections to the structure map $q:E^{univ} \to \bX_0$ is
contractible. Then for any principal $\bX$-bundle $E \to X$ on a
paracompact space $X$, the space of bundle morphisms $E \to
E^{univ}$ is weakly contractible.
\end{thm}

In particular, this applies to the bundle $E \bX$ constructed above.

\begin{proof}
Let us begin with the trivial principal $\bX$-bundle $\bX_1 \to
\bX_0$ from \ref{localmodel}. It is easy to see that the space of
bundle morphisms from $\bX_1 \to \bX_0$ to $E^{univ} \to B^{univ}$
is homeomorphic to the space of sections $s: \bX_0 \to E^{univ}$ of
$q$. Thus, by assumption, the space of bundle morphisms is
contractible in this case. The same argument applies for a trivial
principal bundle over a space different from $\bX_0$.

To achieve the global statement, we apply a trick, which ought to be
standard in the theory of fibre bundles. Let $p:E \to X$ (plus the
additional data) be a principal $\bX$-bundle. Choose an open
covering $(U_{i})_{i \in I}$ of $X$, so that $E|_{U_i}$ is trivial.
For any finite nonempty $S \subset I$, let $U_S := \bigcap_{i \in S}
U_i$. Clearly, $E|_{U_S}$ is trivial as well. Let $\cF_S:=\map_{\bX}
(E|_{U_S}, E^{univ})$ be the space of bundle maps from $E|_{U_S}$ to
$ E^{univ}$. We have seen that $\cF_S $ is weakly contractible. For
any $T \subset S$, there is a restriction map $r_{S}^{T}: \cF_T \to
\cF_S$. Let $\Delta_S$ be the $| S| -1$-dimensional simplex $ \{
\sum_{i \in S} t_i i \in \bR^S | \sum_i t_i = 1; t_i \geq 0 \}$. We
now claim that we can choose maps
\[
c_S: \Delta_S \to \cF_S
\]
such that $r^{T}_{S} \circ c_T = c_S |_{\Delta_T}$ whenever $T
\subset S$. This is done by induction on $|S|$, using the
contractibility of $\cF_S$.

Finally, let $(\lambda_i)_{i \in I}$ be a locally finite partition
of unity subordinate to $(U_i)$. For any point $x \in X$, let $S(x)
\subset I$ be the (finite) set of all $ i \in I$ with $x \in \supp
(\lambda_i)$. The formula
\[
c (y)= c_{S(p(y))} ( \sum_{i \in S(p(y))} \lambda_i ( p(y)) i)  (y)
\]
defines a global bundle morphism. This shows that the space $\cF$ of
bundle morphisms is nonempty. A straightforward version of the
preceding reasoning shows a relative version of it: if $A \subset X$
is a cofibrant inclusion, then any bundle map $E|_{A} \to E^{univ}$
extends to $X$. Thus $\cF$ is connected. A parameterized version of
these arguments shows that for any compact $K$, $\map(K; \cF)$ is
nonempty and connected. Thus $\cF$ is weakly contractible.
\end{proof}

\begin{prop}\label{unweakequ}
The map $\nu_{\bX}:B \bX \to \bX_0 \hq \bX_1$ given by the principal
$\bX$-bundle \ref{univprinbun} is a universal weak equivalence.
\end{prop}

For the proof of \ref{unweakequ}, we shall need a little lemma.

\begin{lem}\label{helplemma}
Let $f: Z \to Y$ be a map between topological spaces. Suppose that
for any map $p:U \to Y$, the space of sections to $f_U:
U \times_Y Z \to U$ is weakly contractible. Then $f$ is a homotopy
equivalence.
\end{lem}

\begin{proof}
We only need that the space is nonempty when $p=\id_Y$ and
path-connected when $p=f$. The assumptions imply that there exists a
section $s: Y \to Z$. We have to show that $s \circ f: Z \to Z$ is
homotopic to the identity. The space of sections to the map $f_Z : Z
\times_Y Z \to Z$; $(z_1,z_2) \mapsto z_1$ is homeomorphic to the
space of maps $g:Z \to Z$ with $f \circ g = f$. The maps $\id_Z$ and
$s \circ f$ both belong to that space and by assumption they are
connected by a homotopy.
\end{proof}

\begin{proof}[Proof of Proposition \ref{unweakequ}] Let $Y$ be a paracompact
space and let $P$ be a principal $\bX$-bundle on $Y$, which gives a
map $Y \to \bX_0 \hq \bX_1$. We have to show that $\nu_{\bX; Y}:Y
\times_{\bX_0 \hq \bX_1} B \bX \to Y$ is a weak homotopy
equivalence. Note that the space of sections to $\nu_{\bX; Y}$ can
be identified with the space of bundle maps $P \to E \bX$, which is
weakly contractible by Theorem \ref{classprinbun}. Likewise, let
$p:Z \to Y$ be any map and let $Q:= p^* P$. The space of sections to
$Z \times_{\bX_0 \hq \bX_1} B \bX \cong Z \times_{Y} Y \times_{\bX_0
\hq \bX_1} B \bX \to Z$ is homeomorphic to the space of bundle maps
from $Q$ to $E \bX$, which is again contractible. Therefore Lemma
\ref{helplemma} can be applied.
\end{proof}

The classical theorem that for a topological group $G$, the set of
isomorphism classes of principal $G$-bundles on a space $Y$ is in
bijective correspondance to the set of homotopy classes of maps $Y
\to BG$ admits a generalization.

\begin{defn}\label{defconcord}
Let $\fX$ and $\fY$ be stacks. Let $h_i : \fX \to \fY$ be two
morphisms. A \emph{concordance} between $h_0$ and $h_1$ is a triple
$(h,\beta_0, \beta_1)$, where $h :\fX \times [0,1]\to \fY $ and
$\beta_i$ is a $2$-isomorphism $h_i \to j_{i}^{*} h$ for $i=0,1$
($j_i$ denotes the inclusion $\fX \cong \fX \times\{ i \} \subset
\fX \times [0,1]$).
\end{defn}

Let $\fX [Y]$ be the set of concordance classes of elements in
$\fX(Y)$. The following is an immediate consequence of Proposition
\ref{unweakequ}.

\begin{cor}\label{lemmconcord}
Let $\fX$ be represented by the groupoid $\bX$. Then there is a
natural bijection $\fX [Y] \cong [ Y; B \bX]$.
\end{cor}

An appropriate relative version is also true; we leave this to the
reader. The last thing we want to show in this section is that
Proposition \ref{unweakequ} actually characterizes $E \bX$ and $B
\bX$ up to homotopy. Let $\bX $ be a groupoid presenting the stack
$\fX$. Let $p : X \to \fX$ be a universal weak equivalence, which
gives rise to a principal $\bX$-bundle $E \to X$. We have seen that
the space of bundle maps $E \to E \bX$ is contractible. Any such
bundle map gives rise to a map $X \to B \bX$. This map is a homotopy
equivalences, which follows immediately from \ref{unweakequ} and
from the $2$-commutativity of the diagram
\[
\xymatrix{   &   B \bX \ar[d]\\
X \ar[r] \ar[ur] & \fX.}
\]

\subsection{Paracompact groupoids}

Later on, we shall need a technical result. Slightly differing from
standard terminology, we say that a topological space $X$ is
\emph{paracompact} if any open covering of $X$ admits a subordinate
locally finite partition of unity. If $X$ is Hausdorff and
paracompact in the usual sense (i.e. any covering admits a locally
finite refinement), then $X$ is paracompact in the present sense,
see \cite{Nag}, p. 427.

We say that a topological groupoid $\bX$ is paracompact and
Hausdorff if all spaces $\bX_n$ of the nerve are paracompact and
Hausdorff. In general, the paracompactness of $\bX_0$ and $\bX_1$
does not imply the paracompactness of $\bX$. However, there are
quite large classes of groupoids which are paracompact and
Hausdorff. Examples of stacks which are presentable by paracompact
and Hausdorff groupoids include
\begin{enumerate}
\item All differentiable stacks modeled on finite-dimensional manifolds\footnote{As usual, we assume manifolds to be second countable.}: If $X$ is a manifold and $X \to \fX$ a representable surjective submersion, then all fibred products $X_n = X \times_{\fX} X \times_{\fX} \ldots X$ are manifolds and hence they are paracompact.
\item More generally, differentiable stacks modeled on Fr\'echet manifolds (under some countability conditions)
\item All topological stacks obtained by taking complex points of algebraic stacks (some finiteness condition is involved).
\item Quotient stacks $X \hq G$ if $G$ and $X$ are metrizable\footnote{It seems unlikely that the paracompactness of both $G$ and $X$ implies the paracompactness of the translation groupoid, because the product of two paracompact spaces is in general not paracompact.}.
\end{enumerate}

Proposition \ref{paracompactness} says that the classifying space $B
\bX = \| N_{\bullet} \bX \|$ of a paracompact and Hausdorff groupoid
is paracompact. This is discussed in the appendix.

\section{Proof of the main results}\label{proofs}

We shall first prove Theorem \ref{mainthm1} and then \ref{mainthm2}.
Let us set up notation. First of all, let $\cS$ be a small
$2$-category and $\cJ: \cS \to  \topstacks$ be a $2$-functor, such
that any stack in the image of $\cJ$ admits a presentation by a
paracompact groupoid. To simplify notation, we assume that $\cJ$ is
the inclusion of a small subcategory; the argument in the general
case is the same. Stacks will be denoted by German letters $\fX,
\fY, \fZ$ and they are tacitly assumed to be in $\cS$; likewise for
morphisms and $2$-morphisms. By the corresponding symbols $\bX,
\bY,\bZ$ we will denote paracompact groupoids representing the
stacks.

For any stack in $\cS$ we choose a presentation by a paracompact
groupoid\footnote{Here we need the axiom of choice, hence we use
that $\cS$ is small.}, denoted by $\varphi_{\fX}: \bX_0 \hq \bX_1
\to \fX$. We can view $\varphi_{\fX}^{-1}$ as a principal $\bX$-bundle on the stack $\fX$, denoted
by $U_{\bX}$. In section \ref{stackprinbun}, we constructed a map
$\nu_{\bX}:B \bX \to \bX_0 \hq \bX_1$. We consider the composition
$\eta_{\bX}= \varphi_{\fX} \circ \nu_{\bX}:B \bX \to \fX$.

Consider a morphism $f: \fX \to \fY$ of stacks. There is a diagram
\[
\xymatrix{
B\bX \ar[d]^{\nu_{\bX}} & B \bY \ar[d]^{\nu_{\bY}} \\
\bX_0 \hq \bX_1 \ar[d]^{\varphi_{\fX}} & \bY_0 \hq \bY_1 \ar[d]^{\varphi_{\fY}} \\
\fX \ar[r]^{f} & \fY,}
\]
and the vertical maps are universal weak equivalences by Proposition \ref{unweakequ}. The
pullback $ \eta_{\fX}^{*} f^* U_{\bY}$ is a principal $\bY$-bundle
on $B \bX$. Let $E_{f}$ be the space of bundle morphisms from $
\eta_{\fX}^{*} f^* U_{\bY}$ to the universal $\bY$-bundle $E \bY \to
B \bY$. By Proposition \ref{paracompactness} $B \bX$ is paracompact
and therefore, by Theorem \ref{classprinbun}, $E_f$ is weakly
contractible. Thus we get a $2$-commutative diagram
\begin{equation}\label{2diagram}
\xymatrix{
E_{f} \times B \bX \ar[d]^{\eta_{\bX} \circ p_2} \ar[r]^-{\epsilon_f} & B \bY \ar[d]^{\eta_{\bY}} \\
\fX \ar[r]^{f} & \fY,
}
\end{equation}
in $\topstacks$; $\epsilon_f$ is the evaluation map and $p_2 $
denotes the projection onto the second factor.

If $f: \fX \to \fY$ and $g: \fY \to \fZ$ are morphisms,
then there is a composition map $c_{g,f}:E_g \times E_f \to E_{g
\circ f}$ which makes the diagram

\[
\xymatrix{
E_{gf} \times B \bX  \ar[r]^{\epsilon_{gf}} & B \bZ \\
E_g \times E_f \times B \bX \ar[r]^-{\id \times \epsilon_f} \ar[u]^{c_{g,f} \times \id} & E_g \times B \bY  \ar[u]^{\epsilon_g}\\
}
\]

commutative. The map $c_{g,f}$ arises in the following way. There is a specified map $E \bY \to U_{\bY}$ of $\bY$-bundles. Thus any element of $E_f$ defines a map $\eta_{\fX}^{*} f^{*} U_{\bY} \to \eta_{\fY}^{*} U_{\bY}$ of $\bY$-bundles and therefore, after application of the morphism $g$ of stacks, a map $\eta_{\fX} f^* g^* U_{\bZ} \to \eta_{\fY}^{*} g^{*} U_{\bZ}$ of $\bZ$-bundles, which can be composed with any element of $E_g$.

The collection of these maps is associative in the
sense that $c_{h \circ g, f} \circ (c_{h,g} \times \id) = c_{h, g
\circ f} \circ (\id \times c_{g,f}):E_h \times E_g \times E_f \to
E_{h \circ g \circ f}$ when $h,g,f$ are composable morphisms. Also,
$\id \in E_{\id}$.

Any $2$-isomorphism $\phi: f \to g$ yields an isomorphism
$\eta_{\bX}^{*} f^* U_{\bY} \cong \eta_{\bX}^{*} g^* U_{\bY}$ and
hence a homeomorphism $\phi^*: E_g \to E_f$. Moreover, the diagram

\begin{equation}\label{commutative}
\xymatrix{
E_g \times B \bX \ar[r]^-{\epsilon_f} \ar[d]^{\phi^* \times \id} & B \bY \ar@{=}[d]\\
E_f \times B \bX \ar[r]^-{\epsilon_g} & B \bY
}
\end{equation}

is commutative.

Now we define an topological category $\widetilde{\tau_{\leq
1}\cS}$. It has the same objects as $\cS$ (and the discrete topology
on the object set). The morphism spaces are

\[
\widetilde{\tau_{\leq 1}\cS}(\fX; \fY) := \coprod_{\cS(\fX; \fY)}
E_f
\]

with the composition described above. The morphism spaces have
contractible components (one for each $1$-morphism in $\cS$). There
is an obvious functor $p:\widetilde{\tau_{\leq 1} \cS} \to
\tau_{\leq 1} (\cS)$.

The classifying space construction determines a continuous functor
\[
\widetilde{\Ho} : \widetilde{\tau_{\leq 1}\cS} \to \ttop,
\]
which sends a stack $\fX$ to the space $B \bX$ and which is the
identity on morphism spaces. The universal weak equivalences
$\eta_{\bX}$ assemble to a pseudo-natural transformation of functors

\[
\eta: \st \circ \widetilde{\Ho} \to \inc \circ p
\]

of functors $\widetilde{\tau_{\leq 1} \cS} \to \tau_{\leq 1}
\topstacks$
(use \ref{2diagram} to verify this). This finishes the proof of Theorem \ref{mainthm1}.\\

Theorem \ref{mainthm2} follows from Theorem \ref{mainthm1} and a
rectification procedure for homotopy-commutative diagrams which is
proven in Nathalie Wahl's paper \cite{Wahl}, following ideas of
Dwyer, Kan and G. Segal. Roughly, she proves that a functor which is
defined up to contractible choices can be strictified. More
precisely, Proposition 2.1. of loc. cit can be reformulated as
follows.

\begin{prop}\cite{Wahl}
Let $\cC $ be a discrete small\footnote{This is an essential
assumption.} category and let $(\widetilde{\cC},p,F)$ be a functor
to $\ttop$ which is defined up to contractible choice. Then there
exists a functor $p_* F: \cC \to \ttop$ and a zig-zag of natural
transformations, which are weak homotopy equivalences on all
objects, connecting $p^* p_* F$ and $F$.
\end{prop}

To finish the proof of Theorem \ref{mainthm2}, take the functor
$\widetilde{\Ho}$ from Theorem \ref{mainthm1} and put
\[
 \Ho(\fX) := p_* \widetilde{\Ho} (\fX).
\]
The additional assertions about $2$-isomorphic and concordant
morphisms of stacks follow from diagram \ref{commutative} and
\ref{lemmconcord}, respectively.

\subsection{Uniqueness of the homotopy type functor}

So far, we have not addressed the question whether our construction
is unique. Here is a uniqueness result which seems to be
satisfactory enough.

Let $\cS$ be a small category of stacks, all of whose objects can be
presented by paracompact groupoids. Let $(\widetilde{\tau_{\leq 1}
\cS},p, \widetilde{\Ho})$ be the homotopy type functor and $\eta:
\st \circ \widetilde{Ho} \to \inc \circ p$ be the natural
transformation constructed in Theorem \ref{mainthm1}. Let
$(\widetilde{\tau_{\leq 1} \cS} \dash ,p \dash ,  \widetilde{\Ho}
\dash )$ be a functor $\tau_{\leq 1} \cS \to \ttop$, defined up to
contractible choice and let $\eta \dash: \st \circ
\widetilde{Ho}\dash \to \inc \circ p \dash$ be a natural
transformation of functors such that $\eta_{\fX}^{\prime}$ is a
universal weak equivalence for any object $\fX$ of $\cS$. These two
functors together yield a functor defined up to contractible choice
on the category $\cS \times \{0,1\}$. Using the arguments from the
proof of Theorem \ref{mainthm1}, one can easily show:

\begin{prop}
Under the circumstances above, there exists a functor
$\overline{\Ho}:\cS \times (0 \to 1) \to \ttop$, defined up to
contractible choice which agrees with $(\widetilde{\tau_{\leq 1}
\cS},p, \widetilde{\Ho})$ on $\cS \times \{ 0 \}$ and with
$(\widetilde{\tau_{\leq 1} \cS}  \dash ,p \dash ,  \widetilde{\Ho}
\dash )$ on $\cS \times \{ 1 \}$. The natural transformations $\eta$
and $\eta \dash$ yield a natural transformation.
\end{prop}

\appendix

\section{Appendix: Point-set-topology of classifying spaces}\label{anhang}

Let $\bX$ be a topological groupoid. There are three different
descriptions of the universal $\bX$-bundle each of which has some
advantages. The first model, which was used in section
\ref{stackprinbun} is the thick geometric realization of the
categories $\bX \downarrow \bX$ and $\bX$. The use of this model
makes the proofs of Lemma \ref{lemma1} and hence Theorem
\ref{classprinbun} particularly transparent.

Another description goes back to Haefliger \cite{Haef}; it is a
generalization of Milnor's classical construction \cite{Mil}. Let
$E^{Mil} \bX$ be the space consisting of sequences $ (t_0 f_0, t_1
f_1, \ldots ) $, where $f_i \in \bX_1$ all have the same target;
$t_i \in [0,1]$, $t_i =0$ for all but finitely many $i \in \bN$,
$\sum_i t_i = 1$. Two sequences $(t_0 f_0, t_1 f_1, \ldots )$ and
$(t_0 \dash f_0 \dash, t_1 \dash f_1 \dash, \ldots )$ are equivalent
if $t_i = t_i \dash$ for each $i$ and $f_i = f_i \dash$ whenever
$t_i \neq 0$. The topology on $E^{Mil} \bX$ is the weakest topology
such that the maps $t_i : E^{Mil} \bX \to [0,1]$ and $f_i :
t_{i}^{-1} (0,1] \to  \bX_1$ are continuous.

To obtain the space $B^{Mil} \bX$, we divide by the following
equivalence relation. Two sequences $(t_0 f_0, t_1 f_1, \ldots )$
and $(t_0 \dash f_0 \dash, t_1 \dash f_1 \dash, \ldots )$ are
identified if $t_i = t_i \dash$ for each $i$ and if there exists an
$f \in \bX_1$ such that $f_i \dash = f_i f$ for all $i$.

A map $E^{Mil} \bX \to \bX_0$ is defined by sending $(t_0 f_0, t_1
f_1, \ldots )$ to the common target of the $f_i$'s, and the action
of $\bX$ is equally easy to define. This description is convenient
for the proof of Proposition \ref{isprinbun}, see \cite{Noo2}.

\begin{lem}\cite{seg}
There are natural homeomorphisms $\|N_{\bullet} \bX \| \cong  B^{Mil} \bX$ and
$\|N_{\bullet} \bX \downarrow \bX \| \cong E^{Mil} \bX $.
\end{lem}

There is a slightly different description of $\|N_{\bullet} \bX \|$
which helps to show Proposition \ref{paracompactness} below. Let
$X_{\bullet}$ be an arbitrary simplicial space. Let $\sk_n
\|X_{\bullet} \| := \bigcup_{k \leq n} X_k \times \Delta^k / \sim $.
Then
\[
 \| X_{\bullet} \| \cong \colim_n \sk_n \|X_{\bullet} \|
\]
and furthermore $\sk_n \|X_{\bullet} \|$ is the pushout
\begin{equation}
\xymatrix{
X_n \times \partial \Delta^n \ar[r] \ar[d] & \sk_{n-1} \|X_{\bullet} \| \ar[d] \\
X_n \times \Delta^n \ar[r] & \sk_n \|X_{\bullet} \|,
}
\end{equation}
where $\sk_{-1} \|X_{\bullet} \|= \emptyset$. Thus, $\sk_n
\|X_{\bullet} \|$ is the double mapping cylinder of a map $X_n
\leftarrow X_n \times \partial \Delta^n \to \sk_{n-1} \|X_{\bullet}
\|$. Therefore $\sk_n \|X_{\bullet} \| \subset \sk_{n+1}
\|X_{\bullet} \|$ is a neighborhood retract, as $X_{n+1} \times
\partial \Delta^{n+1}$ is a neighborhood retract in $X_{n+1} \times
\Delta^{n+1}$.

\begin{prop}\label{paracompactness}
Let $X_{\bullet}$ be a simplicial topological space such that all
$X_n$ are paracompact and Hausdorff. Then the thick realization
$\|X_{\bullet} \|$ is paracompact (but not necessarily Hausdorff).
\end{prop}

A sketch of the proof of Proposition \ref{paracompactness} can be
found in \cite{GH}, p. 14. Here are the details. First note that $
X_n \times \partial \Delta^n $ is paracompact by \cite{Nag}, p.223.
Therefore, the proof of Lemma \ref{paracompactness} is accomplished
by the following two lemmata.

\begin{lem}\label{lemmaeins}
Let $X$ be a paracompact Hausdorff space, $A \subset X$ be a closed
neighborhood retract, $Y$ a paracompact space and let $f:X \to Y$ be
a continuous map. Then $X \cup_A Y $ is paracompact.
\end{lem}

\begin{lem}\label{lemmazwei}
Let $Y$ be the colimit of the sequence $Y_1 \subset Y_2 \subset
\ldots$. Assume that $Y_n$ is paracompact and closed in $Y$; assume
that $Y_n \subset Y_{n+1}$ is a neighborhood retract and assume that
$Y_{n+1} \setminus Y_n$ is paracompact and Hausdorff. Then $Y$ is
paracompact.
\end{lem}

\begin{proof}[Proof of Lemma \ref{lemmaeins}]
Denote the quotient map by $q = q_X \coprod q_Y: X  \coprod Y \to X
\cup_A Y$. Let $B \subset X \cup_A Y$ be an open neighborhood of
$q(Y)$ with a retraction map by $r: B \to Y$. Let $(U_i)_{i \in
\cI}$ be an open covering of $X \cup_A Y$. Let $(\mu_{k})_{k \in
\cK}$ be a locally finite partition of unity on $Y$ which is
subordinate to $q_{Y}^{-1} (U_i)$. Clearly, $r^* \mu_k$ is a locally
finite partition of unity on $B$, but it is not subordinate to
$(U_i)_{i \in \cI} \cap B$. Because $X$ is paracompact and
Hausdorff, it is normal (\cite{Nag}, p. 94, 99) and therefore
Urysohn's lemma applies to it. Namely, we can find a function $c_k :
X  \to [0,1]$ such that $\supp(c_k) \subset q_{X}^{-1}(U_{i(k)})$
and such that $c_k (a) = 1$ if $a \in A$ and $\mu_k (f(a)) > 0$.
Define a function $\nu_k : X \cap_A Y \to [0,1]$ by $\nu_k = \mu_k$
on $q(Y)$ and $\nu_k = c_k r^* \mu_k$ on $q(X )$. Clearly, $\nu_k$
is a locally finite family of functions and the function $\nu=\sum
\nu_k $ is equal to $1$ on $q(Y)$.\\
On the other hand, the space $Z:=\nu^{-1} [0, \frac{2}{3}]$ is a
closed subspace of the paracompact Hausdorff space $X$ and therefore
also paracompact. Then take a partition of unity subordinate to $(Z
\cap U_i)_{i \in \cI}$ and use a bump function $b$ with $b=0$ if
$\nu \leq \frac {1}{3}$ and $b=1 $ if $\nu \geq \frac{2}{3}$ to glue
both partitions of unity together.
\end{proof}

\begin{proof}[Proof of Lemma \ref{lemmazwei}]
To simplify notation, we shall talk about locally finite families
of nonnegative functions without mentioning their individual
members. If $f$ is a locally finite family of nonnegative
functions, we denote the sum of its members by $\sum f$
and the support of $\sum f$ simply by $\supp (f)$.\\
Let $\cU$ be an open covering of $Y$. Any locally finite family $f$
of nonnegative functions on $Y_n$ which is subordinate to $\cU \cap
Y_n$ admits an extension to a locally finite family of nonnegative
functions on all of $Y$, subordinate to $\cU$. This follows from an
iterated application of Urysohn's lemma, using the retraction maps.\\
Nowstart with a partition of unity $\mu^{(1)}$ on $Y_1$ which is
subordinate to $\cU \cap Y_1$. Extend it as above to a locally
finite family on $Y$, also denoted $\mu^{(1)}$. Then choose a
subordinate locally finite family $\mu^{(2)}$ on $Y_2$ whose support
is contained in $\{ \sum \mu^{(1)} \leq \frac{1}{2}  \}$ and which
is equal to $1$ on $\{ \sum \mu^{(1)} \leq \frac{1}{4}  \}$. The sum
$\sum \mu^{(1)} + \sum \mu^{(2)}$ is a locally finite family of functions which is subordinate to $\cU$ and whose support contains $Y_2$.\\
We repeat this process: Assume that locally finite families of
nonnegative functions $\mu^{(1)}, \ldots \mu^{(n)}$ on $Y$ are
chosen, such that the support of $\sum_{k=1}^{n} \sum \mu^{(k)}$
contains $Y_n$. We can define a new locally finite family
$\mu^{(n+1)}$ with support contained in $\{ \sum_{k=1}^{n} \sum
\mu^{(k)} \leq \frac{1}{n+1} \}$ and whose sum is equal to $1$ on
$\{ \sum \mu^{(1)} + \ldots \mu^{(n)} \leq \frac{1}{2(n+1)} \}$.
These conditions guarantee that the union $\cup_{n=1}^{\infty}
\mu^{(n)}$ of these families of functions is locally finite and that
the union of the supports covers $Y$. An obvious formula produces a
partition of unity out of this family.
\end{proof}


\begin{thebibliography}{000000}
\bibitem{AS} M. Atiyah, G. Segal: \emph{Equivariant $K$-theory and completion}. J. Differential Geometry 3 (1969) p. 1-18.
\bibitem{BGNX} K. Behrend, G. Ginot, B. Noohi, P. Xu: \emph{String topology for stacks}, preprint arXiv:0712.3857.
\bibitem{EG} J. Ebert, J. Giansiracusa: \emph{Pontrjagin-Thom maps and the homology of the moduli stack of stable curves} arXiv:0712.0702.
\bibitem{EG2} J. Ebert, J. Giansiracusa: \emph{On the homotopy type of the Deligne-Mumford compactification}, Algebraic \& Geometric Topology 8 (2008) 2049-2062.
\bibitem{FHT} D. Freed, M. Hopkins, C. Teleman: \emph{Loop groups and twisted K-theory I}, preprint, arXiv 0711.1906 (2007).
\bibitem{GH} D. Gepner, A. Henriques: \emph{Homotopy theory of orbispaces}, preprint, arXiv 0701916 (2007).
\bibitem{Haef} A. Haefliger: \emph{Homotopy and integrability}, Manifolds - Amsterdam 1970, Springer Lecture Notes in Mathematics 197 (1971), p. 133-163.
\bibitem{Hein} J. Heinloth: \emph{Some notes on differentiable stacks}, Mathematisches Institut, Seminars 2004/05, Universit\"at G\"ottingen (2005), p. 1-32.
\bibitem{Mil} J. Milnor: \emph{Construction of universal bundles II}, Ann. of Math. 63 (1956), p. 430-436.
\bibitem{Nag} J. Nagata: \emph{Modern general topology}. Second edition. North-Holland Mathematical Library, 33. North-Holland Publishing Co., Amsterdam, (1985).
\bibitem{Noo1} B. Noohi: \emph{Foundations of topological stacks I}, preprint, arXiv 0503247 (2005).
\bibitem{Noo2} B. Noohi: \emph{Homotopy types of stacks}, preprint, arXiv:0808.3799
\bibitem{seg} G. Segal: \emph{Classifying spaces and spectral sequences}.
Inst. Hautes \'Etudes Sci. Publ. Math. IHES, 34 (1968), p. 105-112.
\bibitem{segcat} G. Segal: \emph{Categories and cohomology theories}.
Topology 13 (1974), p. 293-312.
\bibitem{Wahl} N. Wahl: \emph{Infinite loop space structure(s) on the stable mapping class group}, Topology 43 (2004), p. 343-368.
\end{thebibliography}
\end{document}